\documentclass[11pt]{article}
\usepackage[cp1251]{inputenc}
\usepackage[english]{babel}
\usepackage{morefloats}
\usepackage{amsmath}
\usepackage{amsfonts}
\usepackage{amssymb}
\usepackage{graphicx}
\usepackage{multirow}
\usepackage[perpage,para,symbol*]{footmisc}
\usepackage[usenames,dvipsnames]{color}
\newtheorem{theorem}{\textsf{Theorem}}
\newtheorem{lemma}{\textsf{Lemma}}
\newtheorem{proposition}{\textsf{Proposition}}
\newenvironment{proof}[1][\textsf{Proof. }]{\textbf{#1}}{$\square$}
\newenvironment{definition}[1][\textsf{Definition. }]{\textbf{#1}}{}
\newenvironment{example}[1][\textsf{Example. }]{\textbf{#1}}{$\lozenge$}
\newenvironment{note}[1][\textsf{Note. }]{\textbf{#1}}{}
\newenvironment{remark}[1][\textsf{Remark. }]{\textbf{#1}}{}
\def\text{\hbox} \def\newpage{\vfill\break}

\begin{document}

\title{
\date{}
{
\large \textsf{\textbf{Deformation of finite-volume hyperbolic Coxeter polyhedra, limiting growth rates and Pisot numbers}}
}
}
\author{\small Alexander Kolpakov\footnote{supported by the Schweizerischer Nationalfonds SNF no.~200020-121506/1 and no.~200021-131967/1}}
\maketitle

\begin{abstract}\noindent

A connection between real poles of the growth functions for
Coxeter groups acting on hyperbolic space of dimensions three and
greater and algebraic integers is investigated. In particular, a certain 
geometric convergence of fundamental domains for cocompact hyperbolic Coxeter
groups with finite-volume limiting polyhedron provides a
relation between Salem numbers and Pisot numbers. 
Several examples conclude this work.

\medskip
{\textsf{\textbf{Key words}}: hyperbolic Coxeter group, Coxeter polyhedron, growth function, Pisot number.}
\end{abstract}

\parindent=0pt

\section{Introduction}

Since the work of  R.~Steinberg \cite{Steinberg}, growth series for Coxeter groups are
known to be series expansions of certain rational functions. By considering the growth function of a hyperbolic Coxeter
group, being a discrete group generated by a finite set $S$ of reflections in hyperplanes of 
hyperbolic space~$\mathbb{H}^n$, J.W.~Cannon \cite{Cannon1, Cannon2}, P.~Wagreich \cite{Wagreich}, W.~Parry \cite{Parry} and W.~Floyd \cite{Floyd} in the beginning of the 1980's discovered a connection between the real poles of the corresponding growth function and
algebraic integers such as Salem numbers and Pisot numbers for $n=2,3$. In particular, there is a kind of geometric convergence for the fundamental domains of cocompact planar hyperbolic
Coxeter groups giving a geometric interpretation of the convergence of Salem numbers 
to Pisot numbers, the behaviour discovered by R.~Salem~\cite{Salem} much earlier in 1944. This paper provides a generalisation of the result by W.~Floyd \cite{Floyd} to the three-dimensional case (c.f. Theorem~\ref{alg_integers}).

\smallskip
{\bf Acknowledgement.} This work is a part of the author's PhD thesis project supervised by Prof. Ruth Kellerhals. The author is grateful to the Fields Institute for their hospitality and support during the fall 2011 Thematic Program on Discrete Geometry and Applications. The author would like to thank the referees for their valuable comments and suggestions.

\section{Preliminaries}

\textbf{2.1} Let $G$ be a finitely generated group with generating set $S$ providing the pair $(G, S)$. In the following, we often write $G$ for $(G,S)$ assuming $S$ is fixed.
Define the word-norm $\|\cdot\|:G\rightarrow \mathbb{N}$ on $G$ with respect to $S$ by
$\|g\| = \min\,\{n\, |\, g\,\, \mbox{is a product of}\, n\, \mbox{elements from}\, S\cup S^{-1}\}$. Denote by $a_k$ the number of elements in $G$ of word-norm $k$, and put $a_0=1$
as usually done for the empty word. 

\smallskip\begin{definition}
{\it The growth series of the group $G=(G,S)$ with respect to its generating set $S$ is $f(t) := f_{S}(t) = \sum_{k=0}^{\infty} a_k t^k$. }
\end{definition}\smallskip

The series $f(t)$ has positive radius of convergence since $a_k \leq (2\,|S|)^k$. The reciprocal of the radius of convergence is called \textit{the growth rate} $\tau$ of $G$. If $G$ is a Coxeter group with its Coxeter generating set $S$ (c.f. \cite{Humphreys}), then $f(t)$ is a rational function by~\cite{Bourbaki, Steinberg}, called \textit{the growth function} of $G$.

Let $\mathcal{P} \subset \mathbb{H}^n$, $n\geq 2,$ be a finite-volume hyperbolic polyhedron all of whose dihedral angles are submultiples of $\pi$. Such a polyhedron is called a {\it Coxeter polyhedron} and 
gives rise to a discrete subgroup $G = G(\mathcal{P})$ of $\mathrm{Isom}(\mathbb H^n)$ generated by the set $S$
of reflections in the finitely many bounding hyperplanes of $\mathcal{P}$. We call $G = G(\mathcal{P})$ a hyperbolic Coxeter group.
In the following we will study the growth function of $G = (G,S)=G(\mathcal{P})$.

\smallskip
The most important tool in the study of the growth function of a Coxeter group is Steinberg's formula.
\begin{theorem}[R.~Steinberg, \cite{Steinberg}]\label{Steinberg_formula}
Let $G$ be a Coxeter group with generating set $S$. Then
\begin{equation}\label{eqSteinberg}
\frac{1}{f_{S}(t^{-1})} = \sum_{T\in \mathcal{F}}\,\frac{(-1)^{|T|}}{f_T(t)},
\end{equation}
where $\mathcal{F} = \{ T \subseteq S\, |\, \mbox{the subgroup of}\,\, G\, \mbox{generated by}\,\, T\,\, \mbox{is finite} \}$.
\end{theorem}

\smallskip
Consider the group $G = G(\mathcal{P})$ generated by the reflections in the bounding hyperplanes of a hyperbolic Coxeter polyhedron $\mathcal{P}$. 
Denote by $\Omega_k(\mathcal{P})$, $0 \leq k \leq n-1$ the set of all $k$-dimensional faces of $\mathcal{P}$.
Elements in $\Omega_{k}(\mathcal{P})$ for $k=0$, $1$ and $n-1$
are called vertices, edges and facets (or faces, in case $n=3$) of $\mathcal{P}$, respectively.

Observe that all finite subgroups of $G$ are stabilisers of elements $\,F\in \Omega_{k}(\mathcal{P})$ for some $k \geq 0$. By the result of Milnor~\cite{MilnorGrowth}, the growth rate of a hyperbolic group is strictly greater than $1$. Hence, the growth rate of the reflection group $G(\mathcal{P})$ is $\tau > 1$, if $\mathcal{P}$ is compact, and the growth function $f_S(t)$ has a pole in $(0, 1)$.

\medskip
\textbf{2.2} In the context of growth rates we shall look at particular classes of algebraic integers.

\smallskip
\begin{definition}
{\it A Salem number is a real algebraic integer $\alpha > 1$ such that $\alpha^{-1}$ is an algebraic conjugate of $\alpha$ and all the other algebraic conjugates lie on the unit circle of the complex plane. Its minimal polynomial over $\mathbb{Z}$ is called a Salem polynomial. }
\end{definition}

\smallskip
\begin{definition}
{\it A Pisot-Vijayaraghavan number, or a Pisot number for short, is a real algebraic integer $\beta > 1$ such that all the algebraic conjugates of $\beta$ are in the open unit disc of the complex plane. The corresponding minimal polynomial over $\mathbb{Z}$ is called a Pisot polynomial.}
\end{definition}

\smallskip
Recall that a polynomial $P(t)$ is reciprocal if $\tilde{P}(t) = t^{\mathrm{deg}\,P} P(t^{-1})$ equals $P(t)$, and anti-reciprocal if $\tilde{P}(t)$ equals $-P(t)$. The polynomial $\tilde{P}(t)$ itself is called \textit{the reciprocal polynomial} of $P(t)$. 

\medskip
The following result is very useful in order to detect Pisot polynomials.
\begin{lemma}[W.~Floyd, \cite{Floyd}]\label{pisot}
Let $P(t)$ be a monic polynomial with integer coefficients such that $P(0) \neq 0$, $P(1) < 0$, and $P(t)$ is not reciprocal. Let $\tilde{P}(t)$ be the reciprocal polynomial for $P(t)$. Suppose that for every sufficiently large integer $m$, $\frac{t^m P(t) - \tilde{P}(t)}{t - 1}$ is a product of cyclotomic polynomials and a Salem polynomial. Then $P(t)$ is a product of cyclotomic polynomials and a Pisot polynomial.
\end{lemma}

\medskip
The convergence of Salem numbers to Pisot numbers was first discovered and analysed in \cite{Salem}. A geometrical relation between these algebraic integers comes into view as follows. Growth functions of planar hyperbolic Coxeter groups were calculated explicitly in \cite[Section~2]{Floyd}. The main result of \cite{Floyd}  states that the growth rate $\tau$ of a cocompact hyperbolic Coxeter group -- being a Salem number by \cite{Parry} -- converges from below to the growth rate of a finite covolume hyperbolic Coxeter group under a certain deformation process
performed on the corresponding fundamental domains. More precisely, one deforms the given compact Coxeter polygon by decreasing one of its angles $\pi/m$. This process results in pushing one of its vertices toward the ideal boundary $\partial\mathbb H^2$ in such a way that every polygon under this process provides a cocompact hyperbolic Coxeter group.
Therefore, a sequence of Salem numbers $\alpha_m$ given by the respective growth rates $\tau_m$ arises. The limiting Coxeter polygon is of finite area having exactly
one ideal vertex, and the growth rate $\tau_{\infty}$ of the corresponding Coxeter group equals the limit of $\beta = \lim_{m\rightarrow \infty} \alpha_m$ and is a Pisot number.

\medskip
\textbf{2.3} In this work, we study analogous phenomena in the case of spatial hyperbolic Coxeter groups. The next result will play an essential role.

\begin{theorem}[W.~Parry, \cite{Parry}]\label{GF_representation}
Let $\mathcal{P}\subset\mathbb H^3$ be a compact Coxeter polyhedron. The growth function $f(t)$ of $G(\mathcal{P})$ satisfies the identity
\begin{equation}\label{eqParry1}
\frac{1}{f(t^{-1})} = \frac{t-1}{t+1} + \sum_{v \in \Omega_0(\mathcal{P})} g_v(t),
\end{equation}
where
\begin{equation}\label{eqParry2}
g_v(t) = \frac{t(1-t)}{2}\frac{(t^{m_1}-1)(t^{m_2}-1)(t^{m_3}-1)}{(t^{m_1+1}-1)(t^{m_2+1}-1)(t^{m_3+1}-1)}
\end{equation}
is a function associated with each vertex $v\in \Omega_0(\mathcal{P})$; the integers $m_1$, $m_2$, $m_3$ are the Coxeter exponents\footnote{for the definition see, e.g. \cite[Section~3.16, p.~75]{Humphreys}} of the finite Coxeter group $Stab(v)$ (see Table~\ref{tabular_for_GF_representation}). Furthermore, the growth rate $\tau$ of $G(\mathcal{P})$ is a Salem number.
\end{theorem}

\begin{table}[ht]
\begin{center}
\begin{tabular}{|c|c|c|c|}
\hline \multirow{2}{*}{Vertex group $Stab(v)$}&  \multicolumn{3}{|c|}{Coxeter exponents}\\ 
\cline{2-4} &  $m_1$&  $m_2$&  $m_3$\\ 
\hline $\Delta_{2,2,n}$, $n\geq 2$&  1&  1&  $n-1$\\ 
\hline $\Delta_{2,3,3}$&  1&  2&  3\\ 
\hline $\Delta_{2,3,4}$&  1&  3&  5\\ 
\hline $\Delta_{2,3,5}$&  1&  5&  9\\ 
\hline 
\end{tabular} 
\end{center}
\caption{Coxeter exponents}\label{tabular_for_GF_representation}
\end{table}

A rational function $f(t)$, that is not a polynomial, is called reciprocal if $f(t^{-1}) = f(t)$, and anti-reciprocal if $f(t^{-1}) = -f(t)$. In the case of growth functions for Coxeter groups acting cocompactly on $\mathbb H^n$, the following result holds.

\begin{theorem}[R.~Charney, M.~Davis,  \cite{Charney}]\label{reciprocity}
Let $G=(G,S)$ be a Coxeter group acting cocompactly on $\mathbb H^n$. Then, its growth function $f_{S}(t)$ is reciprocal if $n$ is even, and anti-reciprocal if $n$ is odd.
\end{theorem}

For further references on the subject, which treat several general and useful aspects for this
work, we refer to \cite{CLS} and \cite{KP}.

\medskip
The following example illustrates some facts mentioned above.

\smallskip
\begin{example}
Let $\mathcal{D}_n\subset\mathbb H^3\,,\,n\in\mathbb N\,,$ be a hyperbolic dodecahedron with all but one right dihedral angles. The remaining angle along the thickened edge of $D_n$, as shown in Fig.~\ref{dodecahedron}, equals $\frac{\pi}{n+2}$, $n\geq 0$. The initial polyhedron $\mathcal{D}_0$ is known as the L\"{o}bell polyhedron $L(5)$ (see \cite{Vesnin}). As $n \rightarrow \infty$, the sequence of polyhedra tends to a right-angled hyperbolic polyhedron $\mathcal{D}_\infty$ with precisely one vertex at infinity. Let us compute the growth functions and growth rates of $G(\mathcal{D}_n)$, $n\geq 0,$ and $G(\mathcal{D}_{\infty})$.

\begin{figure}[ht]
\begin{center}
\includegraphics* [totalheight=4cm]{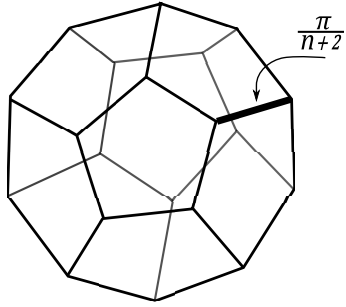}
\end{center}
\caption{The dodecahedron $\mathcal{D}_n \subset \mathbb{H}^3$, $n\geq 0$, with all but one right dihedral angles. The specified angle equals $\frac{\pi}{n+2}$} \label{dodecahedron}
\end{figure}

By Theorem~\ref{GF_representation}, the growth function of $G(\mathcal{D}_n)$, with respect to the generating set $S$ of reflections in the faces of $\mathcal{D}_n$, equals
\begin{equation}\label{eqExample1}
f_n(t) = \frac{(1+t)^3\, (1 + t + \dots + t^{n-1})}{1 - 8 t + 8 t^{n+1} - t^{n+2}}\,\,,
\end{equation}
and similarly
\begin{equation}\label{eqExample2}
f_\infty(t) = \frac{(1+t)^3}{(1 - t)(1 - 8 t)}\,\,.
\end{equation}
Observe that the function (\ref{eqExample1}) is anti-reciprocal, but the function (\ref{eqExample2}) is not.

The computation of the growth rates $\tau_n$, $n\geq 0,$ for $G(\mathcal{D}_n)$ and of the growth rate $\tau_\infty$ for $G(\mathcal{D}_\infty)$ gives 
\begin{equation*}
\tau_0 \approx 7.87298 < \tau_1 \approx 7.98453 < \dots < \tau_{\infty} = 8.
\end{equation*}
Thus, the Salem numbers numbers $\tau_n$, $n\geq 0$, tend from below to $\tau_{\infty}$, which is a Pisot number.
\end{example}

\medskip
Consider a finite-volume polytope $\mathcal{P} \subset \mathbb{H}^n$ and a compact face $F\in \Omega_{n-2}(\mathcal{P})$ with dihedral angle $\alpha_F$. We always suppose that $\mathcal{P}$ is not degenerated (i.e. not contained in a hyperplane). Suppose that there is a sequence of polytopes $\mathcal{P}(k) \subset \mathbb{H}^n$ having the same combinatorial type and the same dihedral angles as $\mathcal{P}=\mathcal{P}(1)$ apart from 
$\alpha_F$ whose counterpart  $\alpha_F(k)$ tends to $0$ as $k \nearrow \infty$. Suppose that the limiting polytope $\mathcal{P}_{\infty}$ exists and has the same number of facets as $\mathcal{P}$. This means the facet $F$, which is topologically a codimension two ball, is contracted to a point, which is a vertex at infinity $v_{\infty} \in \partial \mathbb{H}^n$ of $\mathcal{P}_{\infty}$.
We call this process \textit{contraction of the face $F$ to an ideal vertex}.
\medskip

\begin{remark}
In the case $n=2$, an ideal vertex of a Coxeter polygon $\mathcal{P}_0 \subset \mathbb{H}^2$ comes from ``contraction of a compact vertex'' \cite{Floyd}. This means a vertex $F\in \Omega_{0}(\mathcal{P})$
of some hyperbolic Coxeter polygon $\mathcal{P}$
is pulled towards a point at infinity.
\end{remark}

\medskip
In the above deformation process, the existence of the polytopes $\mathcal{P}(k)$ in hyperbolic space is of fundamental importance. Let us consider the three-dimensional case. Since the angles of hyperbolic finite-volume Coxeter polyhedra are non-obtuse, the theorem by E.M.~Andreev \cite[p.~112, Theorem~2.8]{Vinberg2} is applicable in order to conclude about their existence and combinatorial structure. 

In order to state Andreev's result, recall that a $k$-circuit, $k\geq 3$, is an ordered sequence of faces $F_1,\dots,F_k$ of a given polyhedron $\mathcal{P}$ such that each face is adjacent only to the previous and the following ones, while the last one is adjacent only to the first one, and no three of them share a common vertex.
\begin{theorem}[Vinberg, \cite{Vinberg2}]
Let $\mathcal{P}$ be a combinatorial polyhedron, not a simplex, such that three or four faces meet at every vertex. Enumerate all the faces of $\mathcal{P}$ by $1,\dots, |\Omega_2(\mathcal{P})|$. Let $F_i$ be a face, $E_{ij} = F_i \cap F_j$ an edge, and $V_{ijk}=\cap_{s\in\{i,j,k\}} F_s$ or $V_{ijkl}=\cap_{s\in\{i,j,k,l\}} F_s$ a vertex of $\mathcal{P}$. Let $\alpha_{ij} \geq 0$ be the weight of the edge $E_{ij}$. The following conditions are necessary and sufficient for the polyhedron $\mathcal{P}$ to exist in $\mathbb H^3$ having the dihedral angles $\alpha_{ij}$:
\begin{enumerate}
\item[$(\mathfrak{m}_0)$] $0 < \alpha_{ij} \leq \frac{\pi}{2}$.
\item[$(\mathfrak{m}_1)$] If $V_{ijk}$ is a vertex of $\mathcal{P}$, then $\alpha_{ij}+\alpha_{jk}+\alpha_{ki} \geq \pi$, and if $V_{ijkl}$ is a vertex of $\mathcal{P}$, then $\alpha_{ij}+\alpha_{jk}+\alpha_{kl}+\alpha_{li} = 2\pi$.
\item[$(\mathfrak{m}_2)$] If $F_i$, $F_j$, $F_k$ form a $3$-circuit, then $\alpha_{ij} + \alpha_{jk} + \alpha_{ki} < \pi$.
\item[$(\mathfrak{m}_3)$] If $F_i$, $F_j$, $F_k$, $F_l$ form a $4$-circuit, then $\alpha_{ij}+\alpha_{jk}+\alpha_{kl}+\alpha_{li} < 2\pi$.
\item[$(\mathfrak{m}_4)$] If $\mathcal{P}$ is a triangular prism with bases $F_1$ and $F_2$, then $\alpha_{13}+\alpha_{14}+\alpha_{15}+\alpha_{23}+\alpha_{24}+\alpha_{25} < 3\pi$.
\item[$(\mathfrak{m}_5)$] If among the faces $F_i$, $F_j$, $F_k$, the faces $F_i$ and $F_j$, $F_j$ and $F_k$ are adjacent, $F_i$ and $F_k$ are not adjacent, but concurrent in a vertex $v_{\infty}$, and all three $F_i$, $F_j$, $F_k$ do not meet at $v_{\infty}$, then $\alpha_{ij}+\alpha_{jk} < \pi$.
\end{enumerate}
\end{theorem}

\section{Coxeter groups acting on hyperbolic three-space}

\subsection{Deformation of finite volume Coxeter polyhedra}

Let $\mathcal{P} \subset \mathbb{H}^3$ be a Coxeter polyhedron of finite volume with at least $5$ faces.
Suppose that $k_1, k_2, n, l_1, l_2\ge2\,$ are integers. 

\smallskip
\begin{definition}
{\it An edge $e \in \Omega_1(\mathcal{P})$ is a ridge of type $\langle k_1, k_2, n, l_1, l_2\rangle$ if $e$ 
is bounded and has trivalent vertices $v,w$ such that the dihedral angles at the incident edges are arranged counter-clockwise as follows: 
the dihedral angles along the edges incident to $v$ are  $\frac{\pi}{k_1}$, $\frac{\pi}{k_2}$ and $\frac{\pi}{n}$, the dihedral angle along  the edges incident to $w$ are $\frac{\pi}{l_1}$, $\frac{\pi}{l_2}$ and $\frac{\pi}{n}$. In addition, the faces sharing $e$ are at least quadrangles (see Fig.~\ref{ridge}).}
\end{definition} 

\begin{figure}[ht]
\begin{center}
\includegraphics* [totalheight=3cm]{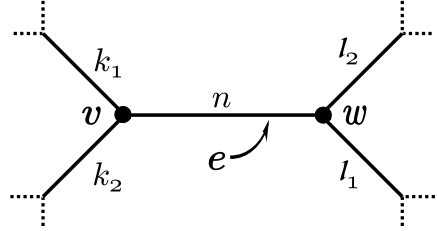}
\end{center}
\caption{A ridge of type $\langle k_1, k_2, n, l_1, l_2 \rangle$} \label{ridge}
\end{figure}

\begin{note}
All the figures Fig.~\ref{resolution_1}-\ref{four_circuit} are drawn according to the following pattern:
only significant combinatorial elements are highlighted (certain vertices, edges and faces), and the remaining  ones are not specified and overall coloured grey. In each figure, the polyhedron is represented by its projection onto one of its supporting planes, and its dihedral angles of the form $\pi/m$ are labelled with $m$.
\end{note}

\begin{proposition}\label{edge_contraction}
Let $\mathcal{P} \subset \mathbb{H}^3$ be a Coxeter polyhedron of finite volume with $|\Omega_2(\mathcal{P})| \geq 5$. If $\mathcal{P}$ has a ridge $e\in \Omega_1(\mathcal{P})$ of type $\langle2,2,n,2,2\rangle$, $n \geq 2$, then $e$ can be contracted to a four-valent ideal vertex.
\end{proposition}
\begin{proof}
Denote by $\mathcal{P}(m)$ a polyhedron having the same combinatorial type and the same dihedral angles as $\mathcal{P}$, except for the angle $\alpha_m = \frac{\pi}{m}$ along $e$. We show that $\mathcal{P}(m)$ exists for all $m \geq n$. Both vertices $v,w$ of $e\in\Omega_1(\mathcal{P}(m))$ are points in $\mathbb{H}^3$, since the sum of dihedral angles at each of them equals $\pi + \frac{\pi}{m}$ for $m \geq n \geq 2$. Thus, condition $\mathfrak{m}_1$ of Andreev's theorem holds. Condition $\mathfrak{m}_0$ is obviously satisfied, as well as conditions $\mathfrak{m}_2$-$\mathfrak{m}_4$, since $\alpha_m \leq \alpha_n$.

\begin{figure}[ht]
\begin{center}
\includegraphics* [totalheight=5cm]{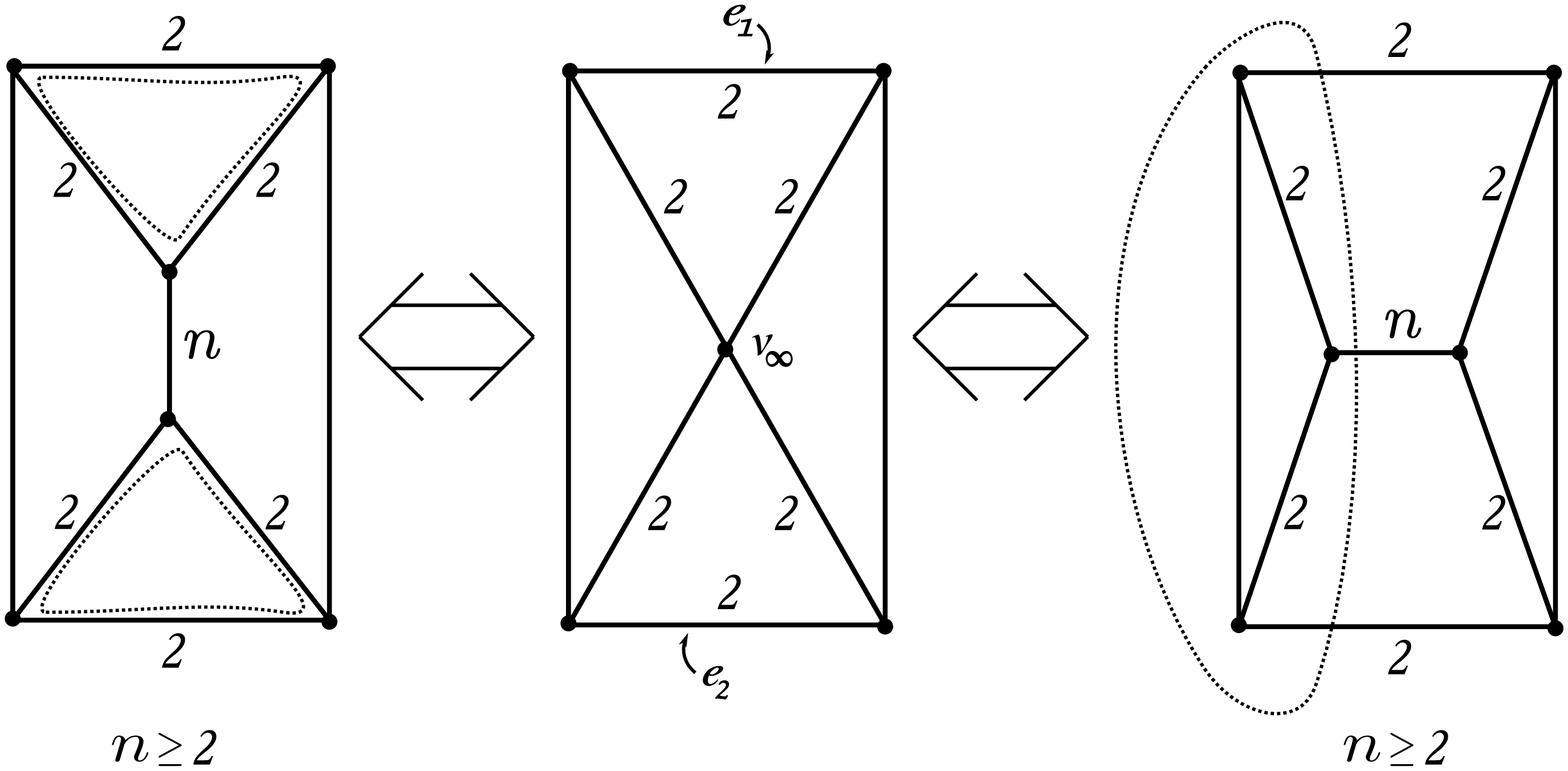}
\end{center}
\caption{Two possible positions of the contracted edge $e$. The forbidden $3$-circuit is dotted and forbidden prism bases are encircled by dotted lines} \label{prism_resolution}
\end{figure}

During the same deformation, the planes intersecting at $e$ become tangent to a point $v_{\infty} \in \partial {\mathbb{H}^3}$ at infinity. The point $v_{\infty}$ is a four-valent ideal vertex with right angles along the incident edges. Denote the resulting polyhedron by $\mathcal{P}_{\infty}$.

Since the contraction process deforms only one edge to a point, no new $3$- or $4$-circuits do appear in $\mathcal{P}_{\infty}$. Hence, for the existence of $\mathcal{P}_{\infty}\subset \mathbb{H}^3$ only condition $\mathfrak{m}_5$ of Andreev's theorem remains to be verified. 
Suppose that condition $\mathfrak{m}_5$ is violated and distinguish the following two cases for the polyhedron $\mathcal{P}$ leading to $\mathcal{P}_{\infty}$ under contraction of the edge $e$.

\smallskip
1.~\textit{$\mathcal{P}$ is a triangular prism}. There are two choices of the edge $e\in \Omega_1(\mathcal{P})$, that undergoes contraction to $v\in\Omega_{\infty}(\mathcal{P}_{\infty})$, as shown in Fig.~\ref{prism_resolution} on the left and on the right. Since  $\mathcal{P}_{\infty}$ is a Coxeter polyhedron, the violation of $\mathfrak{m}_5$ implies that the dihedral angles along the edges $e_1$ and $e_2$ have to equal $\pi/2$. But then, either condition $\mathfrak{m}_2$ or $\mathfrak{m}_4$ is violated, depending on the position of the edge $e$.

2.~\textit{Otherwise}, the two possible positions of the edge $e$ in Fig.~\ref{resolution_1} and Fig.~\ref{resolution_2}. The dihedral angles along the top and bottom edges are right, since $\mathfrak{m}_5$ is violated after contraction.

\begin{figure}[ht]
\begin{center}
\includegraphics* [totalheight=4.5cm]{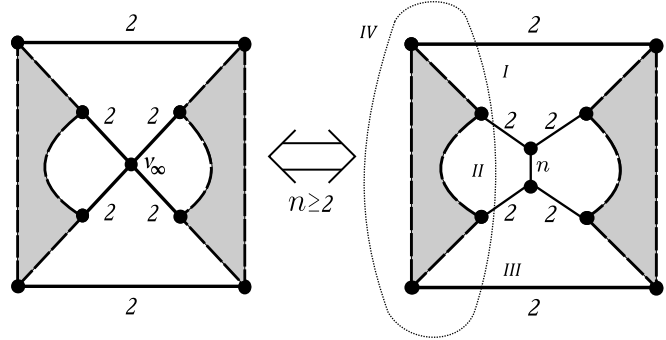}
\end{center}
\caption{The first possible position of the contracted edge $e$. The forbidden $4$-circuit is dotted. Face $IV$ is at the back of the picture} \label{resolution_1}
\end{figure}

2.1~\textit{Consider} the polyhedron $\mathcal{P}$ in Fig.~\ref{resolution_1} on the right. Since $\mathcal{P}$ is not a triangular prism, we may suppose (without loss of generality) that the faces $I$, $II$, $III$, $IV$ in the picture are separated by at least one more face lying in the left grey region. But then, the faces $I$, $II$, $III$ and $IV$ of $\mathcal{P}$ form a $4$-circuit violating condition $\mathfrak{m}_3$ of Andreev's theorem.

\begin{figure}[ht]
\begin{center}
\includegraphics* [totalheight=4.5cm]{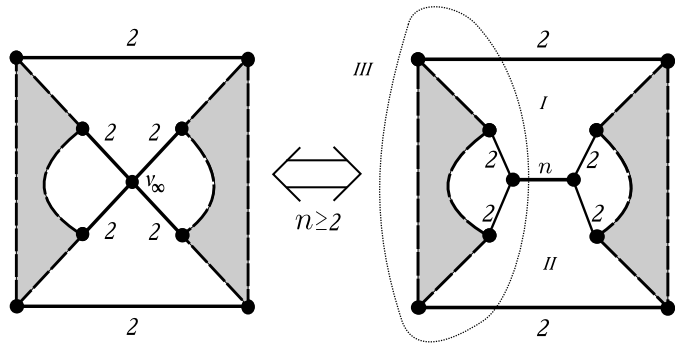}
\end{center}
\caption{The second possible position of the contracted edge $e$. The forbidden $3$-circuit is dotted. Face $III$ is at the back of the picture} \label{resolution_2}
\end{figure}

2.2~\textit{Consider} the polyhedron $\mathcal{P}$ on the right in Fig.~\ref{resolution_2}. As before, we may suppose that the faces $I$, $II$, $III$ form a $3$-circuit. This circuit violates condition $\mathfrak{m}_2$ of Andreev's theorem for $\mathcal{P}$.

Thus, the non-existence of $\mathcal{P}_{\infty}$ implies the non-existence of $\mathcal{P}$, and one arrives at a contradiction. 
\end{proof}

\medskip
\begin{note}
Proposition~\ref{edge_contraction} describes the unique way of ridge contraction. Indeed,
there is only one infinite family of distinct spherical Coxeter groups representing $Stab(v)$, where $v$ is a vertex of the ridge $e$, and this one is $\Delta_{2,2,n}$, $n\geq 2$. One may compare the above limiting process for hyperbolic Coxeter polyhedra with the limiting process for orientable hyperbolic 3-orbifolds from~\cite{Dunbar}.
\end{note}

\begin{proposition}\label{four_valent_vertex}
Let $\mathcal{P} \subset \mathbb{H}^3$ be a Coxeter polyhedron of finite volume
with at least one four-valent ideal vertex $v_{\infty}$. Then
there exists a sequence of finite-volume Coxeter polyhedra $\mathcal{P}(n)\subset \mathbb{H}^3$ having the same combinatorial type and dihedral angles as $\mathcal{P}$ except for a ridge of type $\langle2,2,n,2,2\rangle$, with $n$ sufficiently large, giving rise to the vertex $v_{\infty}$ under contraction.
\end{proposition}
\begin{proof}
Consider the four-valent ideal vertex $v_{\infty}$ of $\mathcal{P}$ and replace $v_{\infty}$ by an edge $e$ in one of the two ways as shown in Fig.~\ref{resolution} while keeping the remaining combinatorial elements of $\mathcal{P}$ unchanged. Let the dihedral angle along $e$ be equal to $\frac{\pi}{n}$, with $n\in\mathbb N$ sufficiently large. We denote this new polyhedron by $\mathcal{P}(n)$. The geometrical meaning of the ``edge contraction'' - ``edge insertion'' process is illustrated in Fig.~\ref{resolutiongeom}. We have to verify the existence of $\mathcal{P}(n)$ in $\mathbb{H}^3$. Conditions $\mathfrak{m}_0$ and $\mathfrak{m}_1$ of Andreev's theorem are obviously satisfied for $\mathcal{P}(n)$. Condition $\mathfrak{m}_5$ is also satisfied since $n$ can be taken large enough.

\begin{figure}[ht]
\begin{center}
\includegraphics* [totalheight=3cm]{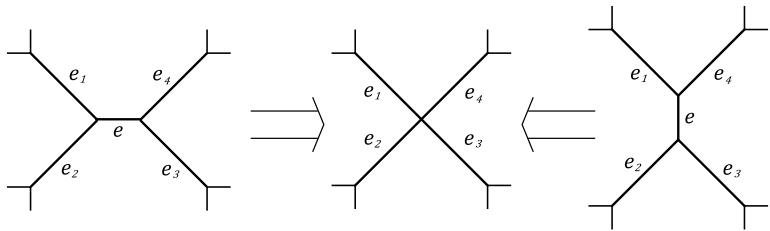}
\end{center}
\caption{Two possible ridges resulting in a four-valent vertex under contraction}\label{resolution}
\end{figure}

Suppose that one of the remaining conditions of Andreev's theorem is violated. The inserted edge $e$ of $\mathcal{P}(n)$ might appear in a new $3$- or $4$-circuit not present in $\mathcal{P}$ so that several cases are possible.

\begin{figure}[ht]
\begin{center}
\includegraphics* [totalheight=2cm]{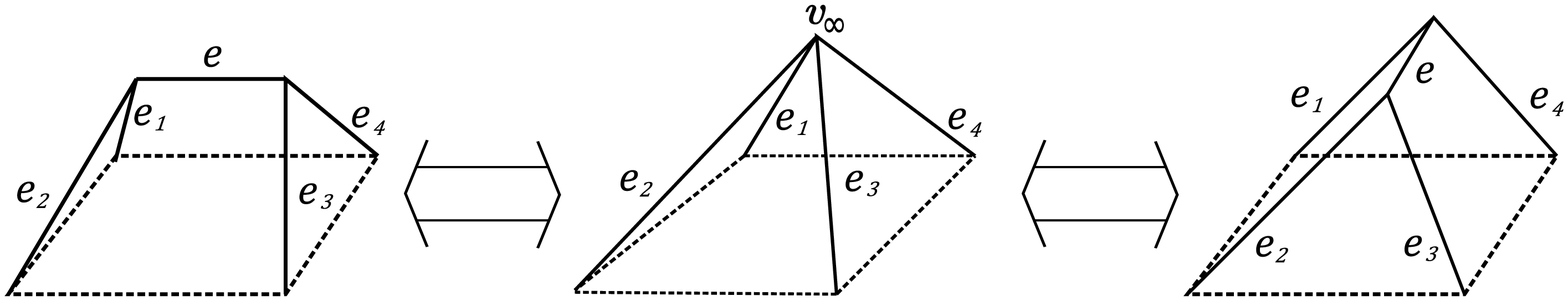}
\end{center}
\caption{Pushing together and pulling apart the supporting planes of polyhedron's faces results in an ``edge contraction''- ``edge insertion'' process} \label{resolutiongeom}
\end{figure}

1.~\textit{$\mathcal{P}(n)$ is a triangular prism}. The polyhedron $\mathcal{P}(n)$ violating condition $\mathfrak{m}_2$ of Andreev's theorem is illustrated in Fig.~\ref{prism_resolution} on the right. Since $\mathcal{P}(n)$ is Coxeter, the $3$-circuit depicted by the dashed line comprises the three edges in the middle, with dihedral angles $\frac{\pi}{n}$, $\frac{\pi}{2}$ and $\frac{\pi}{2}$ along them. Contracting the edge $e$ back to $v_{\infty}$, we observe that condition $\mathfrak{m}_5$ for the polyhedron $\mathcal{P}$ does not hold. 

Since there are no $4$-circuits, the only condition of Andreev's theorem for $\mathcal{P}(n)$, which might be yet violated, is $\mathfrak{m}_4$. This case is depicted in Fig.~\ref{prism_resolution} on the left. A similar argument as above leads to a contradiction.

\smallskip
2.~\textit{Otherwise}, we consider the remaining unwanted cases, when either condition $\mathfrak{m}_2$ or condition $\mathfrak{m}_3$ is violated.

\smallskip
2.1~\textit{Case of a $3$-circuit}. In Fig.~\ref{three_circuit_1} and Fig.~\ref{three_circuit_2}, we illustrate two possibilities to obtain a $3$-circuit in $\mathcal{P}(n)$ for all $n$ sufficiently large, which violates condition $\mathfrak{m}_2$ of Andreev's theorem. The faces of the $3$-circuit are indicated by $I$, $II$ and $III$. In  Fig.~\ref{three_circuit_1}, the edge $e$ is ``parallel'' to the circuit, meaning that $e$ belongs to precisely one of the faces $I$, $II$ or $III$.

\begin{figure}[ht]
\begin{center}
\includegraphics* [totalheight=5cm]{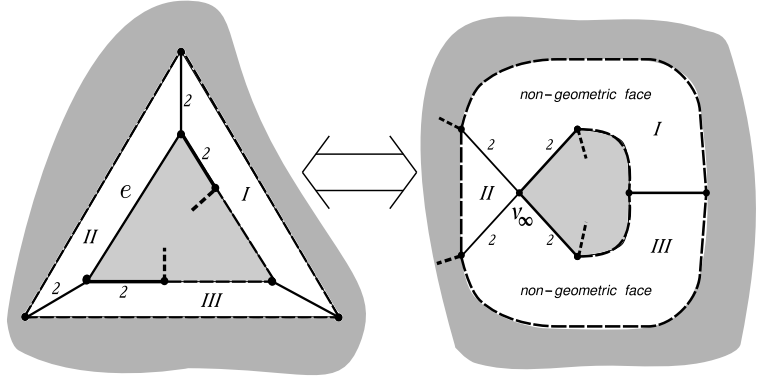}
\end{center}
\caption{Forbidden $3$-circuit: the first case} \label{three_circuit_1}
\end{figure}

\begin{figure}[ht]
\begin{center}
\includegraphics* [totalheight=4.5cm]{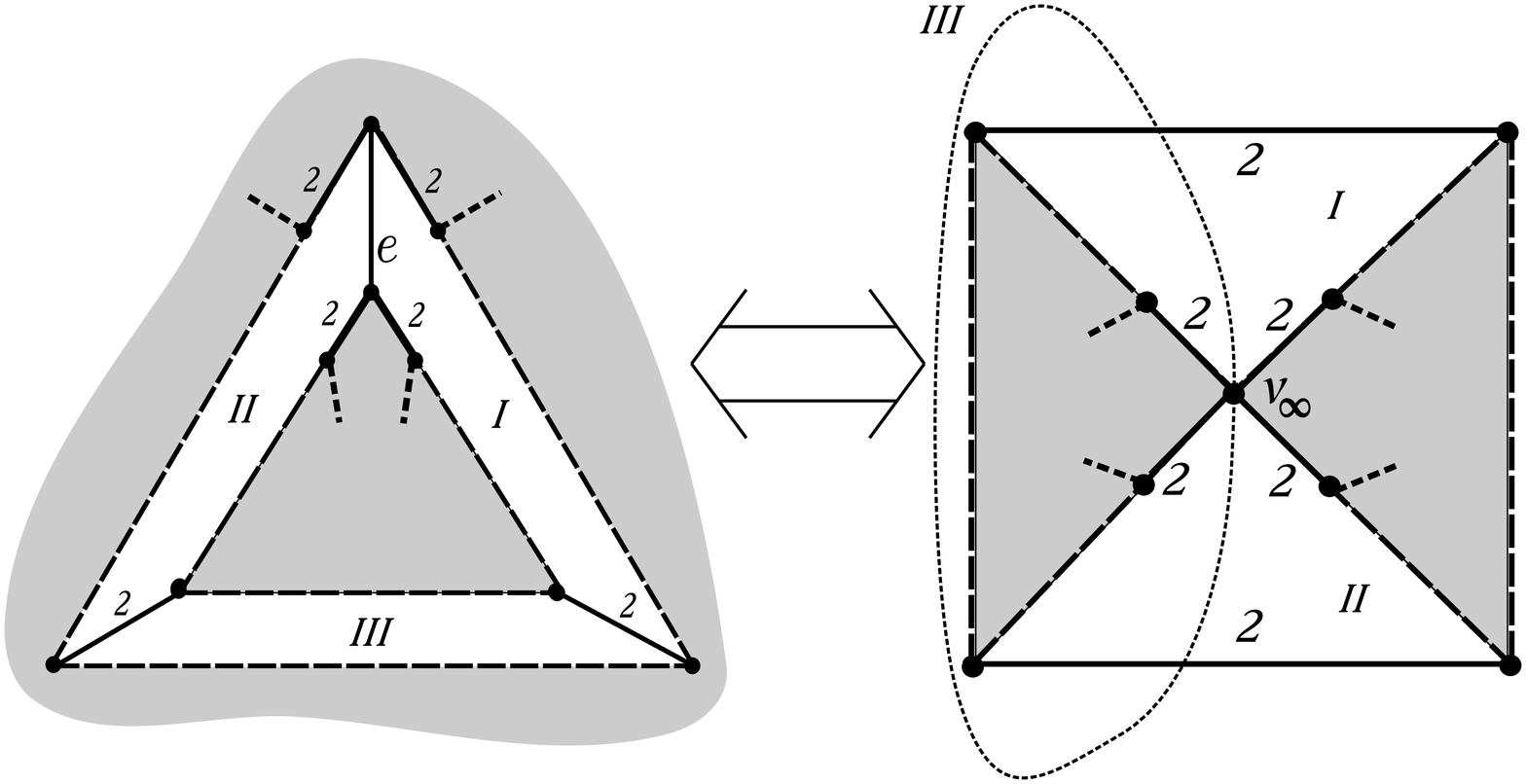}
\end{center}
\caption{Forbidden $3$-circuit: the second case. The forbidden circuit going through the ideal vertex is dotted. Face $III$ is at the back of the picture} \label{three_circuit_2}
\end{figure}

In Fig.~\ref{three_circuit_2}, the edge $e$ is ``transversal'' to the circuit, meaning that $e$ is the intersection of precisely two of the faces $I$, $II$ or $III$. Contracting $e$ back to $v_{\infty}$ leads to an obstruction for the given polyhedron $\mathcal{P}$ to exist, as illustrated in Fig.~\ref{three_circuit_1} and Fig.~\ref{three_circuit_2} on the right. The polyhedron $\mathcal{P}$ in Fig.~\ref{three_circuit_1} has two non-geometric faces, namely $I$ and $III$, having in common precisely one edge and the vertex $v_{\infty}$ disjoint from it. 
The polyhedron $\mathcal{P}$ in Fig.~\ref{three_circuit_2} violates condition $\mathfrak{m}_5$ of Andreev's theorem because of the triple, that consists of the faces $I$, $II$ and $III$ (in Fig.~\ref{three_circuit_2} on the right, the face $III$ is at the back of the picture).

\smallskip
2.2~\textit{Case of a $4$-circuit}. First, observe that the sum of dihedral angles along the edges involved in a $4$-circuit transversal to the edge $e$ does not exceed $\frac{3\pi}{2} + \frac{\pi}{n}$, and therefore is less than $2\pi$ for all $n>2$. This means condition $\mathfrak{m}_3$ of Andreev's theorem is always satisfied for $n$ sufficiently large. 

\begin{figure}[ht]
\begin{center}
\includegraphics* [totalheight=4.5cm]{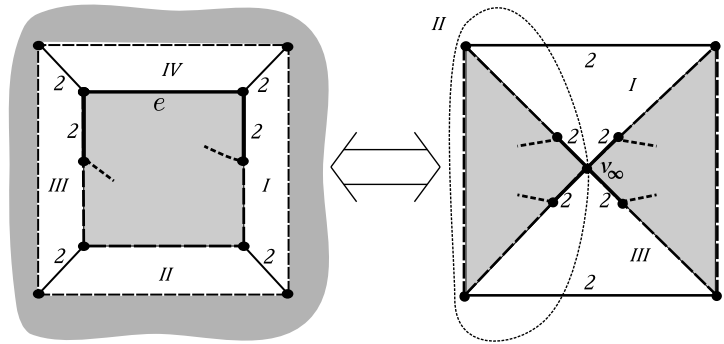}
\end{center}
\caption{Forbidden $4$-circuit. The forbidden circuit going through the ideal vertex is dotted. Face $II$ is at the back of the picture} \label{four_circuit}
\end{figure}

Finally, a $4$-circuit parallel to the edge $e$ in $\mathcal{P}(n)$ is illustrated in Fig.~\ref{four_circuit}. The faces in this $4$-circuit are indicated by $I$, $II$, $III$, $IV$. Suppose that the $4$-circuit violates condition $\mathfrak{m}_3$. Contracting $e$ back to $v_{\infty}$ (see Fig.~\ref{four_circuit} on the right) leads to a violation of $\mathfrak{m}_5$ for $\mathcal{P}$ because of the circuit, that consists of the faces $I$, $II$ and $III$ (in Fig.~\ref{four_circuit} on the right, the face $II$ is at the back of the picture). 
\end{proof}

\medskip
\begin{note}
The statements of Proposition~\ref{edge_contraction} and Proposition~\ref{four_valent_vertex} are essentially given in \cite[p.~238]{Vinberg2} without proof. In the higher-dimensional case, no codimension two face contraction is possible. Indeed, the contraction process produces a finite-volume polytope $\mathcal{P}_{\infty} \subset \mathbb{H}^n$, $n \geq 4$, whose volume is a limit point for the set of volumes of $\mathcal{P}(k)\subset \mathbb{H}^n$ as $k \rightarrow \infty$.  But, by the theorem of H.-C. Wang \cite[Theorem~3.1]{Vinberg2}, the set of volumes of Coxeter polytopes in $\mathbb{H}^n$ is discrete if $n \geq 4$.
\end{note}

\section{Limiting growth rates of Coxeter groups acting on $\mathbb{H}^{3}$}

The result of this section is inspired by W.~Floyd's work \cite{Floyd} on planar hyperbolic Coxeter groups. We consider a sequence of compact polyhedra $\mathcal{P}(n)\subset\mathbb{H}^{3}$ with a ridge of type $\langle2,2,n,2,2\rangle$ converging, as $n \rightarrow \infty$, to a polyhedron $\mathcal{P}_{\infty}$ with a single four-valent ideal vertex. According to \cite{Parry}, all the growth rates of the corresponding reflection groups $G(\mathcal{P}(n))$ are Salem numbers. Our aim is to show that the limiting growth rate is a Pisot number.

\smallskip
The following definition will help us to make the technical proofs more transparent 
when studying the analytic behaviour of growth functions.

\medskip
\begin{definition}
{\it For a given Coxeter group $G$ with generating set $S$ and growth function $f(t) = f_{S}(t)$, define $F(t) = F_{S}(t) := \frac{1}{f_{S}(t^{-1})}$.}
\end{definition}

\begin{proposition}\label{GE_ideal}
Let $\mathcal{P}_{\infty} \subset \mathbb{H}^3$ be a finite-volume Coxeter polyhedron with at least one four-valent ideal vertex obtained from a sequence of finite-volume Coxeter polyhedra $\mathcal{P}(n)$ by contraction of a ridge of type $\langle 2,2,n,2,2 \rangle$ as $n\rightarrow\infty$. Denote by $f_n(t)$ and $f_\infty(t)$ the growth functions of $G(\mathcal{P}(n))$ and $G(\mathcal{P}_\infty)$, respectively. Then
\begin{equation*}
\frac{1}{f_n(t)}-\frac{1}{f_{\infty}(t)} = \frac{t^n}{1-t^n}\left(\frac{1-t}{1+t}\right)^2.
\end{equation*}
Moreover, the growth rate $\tau_n$ of $G(\mathcal{P}(n))$ converges to the growth rate $\tau_{\infty}$ of $G(\mathcal{P}_{\infty})$ from below.
\end{proposition}
\begin{proof}
We calculate the difference of $F_n(t)$ and $F_{\infty}(t)$ by means of equation~(\ref{eqSteinberg}). In fact, this difference is caused only by the stabilisers of the ridge $e \in \Omega_1(\mathcal{P}(n))$ and of its vertices $v_i \in \Omega_0(\mathcal{P}(n))$, $i=1,2$. Let $[k] := 1+\dots+t^{k-1}$. Here $Stab(e) \simeq D_n$, the dihedral group of order $2n$, and $Stab(v_i) \simeq \Delta_{2,2,n}$. The corresponding growth functions are given by $f_e(t) = [2][n]$ and $f_{v_i}(t) = [2]^2[n]$, $i=1,2$ (see \cite{Solomon}). Thus
\begin{equation}\label{dgrowth}
F_n(t) - F_{\infty}(t) = \frac{1}{f_{e}(t)} -\frac{1}{f_{v_1}(t)} - \frac{1}{f_{v_2}(t)} = \frac{1}{t^n - 1} \left(\frac{t-1}{t+1}\right)^2.
\end{equation}
Next, perform the substitution $t \rightarrow t^{-1}$ on (\ref{dgrowth}) and use the relation between $F_n(t)$, $F_{\infty}(t)$ and their counterparts $f_n(t)$ and $f_{\infty}(t)$ according to the definition above. As a result, we obtain the desired formula, which yields $\frac{1}{f_n(t)}-\frac{1}{f_{\infty}(t)} > 0$ for $t \in (0,1)$.

Consider the growth rates $\tau_n$ and $\tau_{\infty}$ of $G(\mathcal{P}(n))$ and $G(\mathcal{P}_{\infty})$. The least positive pole of $f_n(t)$ is the least positive zero of $\frac{1}{f_n(t)}$, and $f_n(0) = 1$. Similar statements hold for $f_{\infty}(t)$. Hence, by the inequality above and by the definition of growth rate, we obtain  $\tau^{-1}_n > \tau^{-1}_{\infty}$, or $\tau_n < \tau_{\infty}$, as claimed.

Finally, the convergence  $\tau_n \rightarrow \tau_{\infty}$ as $n\rightarrow \infty$
follows from the convergence $\frac{1}{f_n(t)}-\frac{1}{f_{\infty}(t)} \rightarrow 0$ on $(0,1)$, due to the first part of the proof.
\end{proof}

\medskip
\begin{note}
Given the assumptions of Proposition~\ref{GE_ideal}, the volume of $\mathcal{P}(n)$ is less than that of $\mathcal{P}_{\infty}$ by Schl\"{a}fli's volume differential formula \cite{MilnorSchlaefli}. Thus, growth rate and volume are both increasing under contraction of a ridge.
\end{note}

\medskip
Consider two Coxeter polyhedra $\mathcal{P}_1$ and $\mathcal{P}_2$ in $\mathbb H^3$
having the same combinatorial type and dihedral angles
except for the respective ridges $H_1 = \langle k_1,k_2, n_1, l_1, l_2 \rangle$ and $H_2 = \langle k_1, k_2, n_2, l_1, l_2 \rangle$.

\smallskip
\begin{definition}
{\it We say that $H_1 \prec H_2$ if and only if $n_1 < n_2$.}
\end{definition}

\medskip
The following proposition extends  Proposition~\ref{GE_ideal} to a more general context.

\begin{proposition}\label{GE_growth}
Let $\mathcal{P}_1$ and $\mathcal{P}_2$ be two compact hyperbolic Coxeter polyhedra having the same combinatorial type and dihedral angles except for an edge of ridge type $H_1$ and $H_2$, respectively. If $H_1 \prec H_2$, then the growth rate of $G(\mathcal{P}_1)$ is less than the growth rate of $G(\mathcal{P}_2)$.
\end{proposition}

\begin{proof}
Denote by $f_1(t)$ and $f_2(t)$
the growth functions of $G(\mathcal{P}_1)$ and $G(\mathcal{P}_2)$, respectively. As before, we will show that $\frac{1}{f_1(t)}-\frac{1}{f_2(t)} \geq 0$ on $(0, 1)$. 

Without loss of generality, we may suppose the ridges $H_i$ to be of type $\langle k_1, k_2, n_i, l_1, l_2 \rangle$, $i=1,2$, up to a permutation of the sets $\{k_1, k_2\}$, $\{l_1, l_2\}$ and $\{\{k_1, k_2\}, \{l_1, l_2\}\}$. By means of Table~\ref{tabular_for_GF_representation} showing all the finite triangle reflection groups, all admissible ridge pairs can be determined. We collected them in Tables~\ref{tabular_for_GE_growth1}--\ref{tabular_for_GE_growth2} at the last pages of the paper. The rest of the proof, starting with the computation of $\frac{1}{f_1(t)}-\frac{1}{f_2(t)}$ in accordance with Theorem~\ref{GF_representation}, equations~(\ref{eqParry1}) and (\ref{eqParry2}), follows by analogy to Proposition~\ref{GE_ideal}.
\end{proof}

\medskip
From now on $\mathcal{P}(n)$ always denotes a sequence of compact polyhedra in $\mathbb H^3$
having a ridge of type $\langle2,2,n,2,2\rangle$, with $n$ sufficiently large, that converges to a polyhedron $\mathcal{P}_{\infty}$ with a single four-valent ideal vertex. The corresponding growth functions for the groups $G(\mathcal{P}_n)$ and $G(\mathcal{P}_{\infty})$ are denoted by $f_n(t)$ and $f_{\infty}(t)$. As above, we will work with the functions $F_n(t)$ and $F_{\infty}(t)$. By Theorem~\ref{reciprocity}, both $f_n(t)$ and $F_n(t)$ are anti-reciprocal rational functions. 

\smallskip
Consider the denominator of the right-hand side of Steinberg's formula (c.f. Theorem~\ref{Steinberg_formula}). According to
\cite[Section 5.2.2]{CLS} and \cite[Section 2.1]{KP}, we introduce the following concept.

\medskip
\begin{definition}
{\it The least common multiple of the polynomials $f_T(t)\,,\,T \in \mathcal{F}\,,$ is called the virgin form of the numerator of  ${f_{S}(t^{-1})} $}.
\end{definition}

\smallskip
The next result describes the virgin form of the denominator of $F_{\infty}(t)$.
\begin{proposition}\label{GF_ideal}
Let $\mathcal{P}_{\infty}\subset\mathbb H^3$ be a polyhedron of finite volume with a single four-valent ideal
vertex. Then the
function $F_{\infty}(t)$ related to the Coxeter group $G(\mathcal{P}_{\infty})$ is given by
\begin{equation*}
F_{\infty}(t) = \frac{t(t-1)P_{\infty}(t)}{Q_{\infty}(t)},
\end{equation*}
where $Q_{\infty}(t)$ is a product of cyclotomic polynomials, $\mathrm{deg}\,Q_{\infty}(t) - \mathrm{deg}\,P_{\infty}(t) = 2$, and $P_{\infty}(0) \neq 0$, $P_{\infty}(1) < 0$.
\end{proposition}
\begin{proof}
The denominator of $F_{\infty}(t)$ in its virgin form is a product of cyclotomic polynomials $\Phi_k(t)$ with $k\geq 2$.  By means of the equality $F_{\infty}(1) = \chi(G(\mathcal{P}_{\infty}))=0$ (see \cite{Serre}),
the numerator of $F_{\infty}(t)$ is divisible by $t-1$. Moreover, by \cite[Corollary 5.4.5]{CLS}, the growth function $f_{\infty}(t)$
for $G(\mathcal{P}_{\infty})$ has a simple pole at infinity. This means $F_{\infty}(t)$ has a simple zero at $t=0$, so that the numerator of $F_{\infty}(t)$ has the form $t(t-1)P_{\infty}(t)$, where $P_{\infty}(t)$ is a polynomial such that $P_{\infty}(0) \neq 0$. The desired equality $\mathrm{deg}\,Q_{\infty}(t) - \mathrm{deg}\,P_{\infty}(t) = 2$ follows from $f_{\infty}(0) = 1$.

The main part of the proof is to show that $P_{\infty}(1) < 0$. By the above, $\frac{\mathrm{d}F_{\infty}}{\mathrm{d}t}(1) = \frac{P_{\infty}(1)}{Q_{\infty}(1)}$ whose denominator is a product of cyclotomic polynomials $\Phi_k(t)$ with $k\geq 2$ evaluated at $t=1$. Hence $Q_{\infty}(1) > 0$, and it suffices to prove that $\frac{\mathrm{d}F_{\infty}}{\mathrm{d}t}(1)<0$.

Consider a sequence of combinatorially isomorphic compact polyhedra $\mathcal{P}(n)$ in $\mathbb H^3$
having a ridge of type $\langle2,2,n,2,2\rangle$ and converging to $\mathcal{P}_{\infty}$. By Proposition~\ref{GE_ideal},
\begin{equation*}
\frac{\mathrm{d}F_n}{\mathrm{d}t}(1) - \frac{\mathrm{d}F_{\infty}}{\mathrm{d}t}(1) = \frac{1}{4n}~.
\end{equation*}
In order to show $\frac{\mathrm{d}F_{\infty}}{\mathrm{d}t}(1)<0$, it is enough to
prove that $\frac{\mathrm{d}F_n}{\mathrm{d}t}(1)<0$
for $n$ large enough. To this end, we consider
the following identity which is a consequence of Theorem~\ref{GF_representation}, equations (\ref{eqParry1})-(\ref{eqParry2}):
\begin{equation*}
\frac{\mathrm{d}F_n}{\mathrm{d}t}(1) = \frac{1}{2} + \sum_{v \in \Omega_0(\mathcal{P}(n))}\frac{\mathrm{d}g_v}{\mathrm{d}t}(1)~.
\end{equation*}
In Table~\ref{tabular_for_GF_ideal} at the last pages, we list all possible values $\frac{\mathrm{d}g_v}{\mathrm{d}t}(1)$ depending on the subgroup $Stab(v)$ of $G(\mathcal{P}(n))$. It follows that $\frac{\mathrm{d}g_v}{\mathrm{d}t}(1) \leq -\frac{1}{16}$ for every $v\in\Omega_0(\mathcal{P}(n))$. Provided $|\Omega_0(\mathcal{P}(n))|\geq 10$, we obtain the estimate $\frac{\mathrm{d}F_n}{\mathrm{d}t}(1) \leq -\frac{1}{8}$. 

Consider the remaining cases $~5\le|\Omega_0(\mathcal{P}(n))|< 10$. By the simplicity of the polyhedron  $\mathcal{P}(n)$, we have that $2|\Omega_1(\mathcal{P}(n))| = 3|\Omega_0(\mathcal{P}(n))|$. Therefore $|\Omega_0(\mathcal{P}(n))|$ is an even number. Hence, the only cases consist of $|\Omega_0(\mathcal{P}(n))| = 8$, meaning that $\mathcal{P}(n)$ is either a combinatorial cube or a doubly truncated tetrahedron (see Fig.~\ref{polyhedra8vert}), and $|\Omega_0(\mathcal{P}(n))| = 6$, meaning that $\mathcal{P}(n)$ is a combinatorial triangular prism. 

\begin{figure}[ht]
\begin{center}
\includegraphics* [totalheight=3cm]{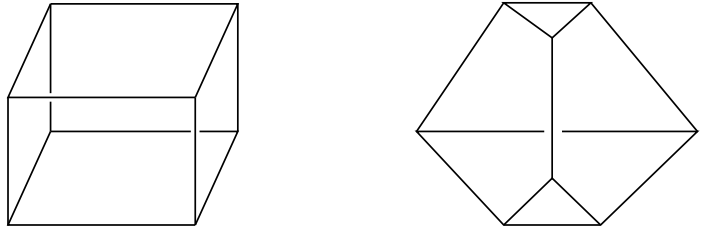}
\end{center}
\caption{Simple polyhedra with eight vertices} \label{polyhedra8vert}
\end{figure}

In the former case, not all the vertices of $\mathcal{P}(n)$ have their stabilizers isomorphic to $\Delta_{2,2,2}$, since $\mathcal{P}(n)$ is a non-Euclidean cube or a non-Euclidean tetrahedron with two ultra-ideal vertices. Then Table ~\ref{tabular_for_GF_ideal} provides the desired inequality $\frac{\mathrm{d}F_n}{\mathrm{d}t}(1) < 0$. The latter case requires a more detailed consideration. We use the list of hyperbolic Coxeter triangular prisms given by~\cite{Kaplinskaya, Vinberg1}\footnote{for how to read hyperbolic Coxeter diagrams, we refer to \cite[Ch.~5, \S 1.3]{Vinberg2}}. These prisms have one base orthogonal to all adjacent faces. More general Coxeter prisms arise by gluing the given ones along their orthogonal bases, if the respective planar angles coincide. 

Among all triangular Coxeter prisms, we depict in Fig.~\ref{prismcontraction1}-\ref{prismcontraction3} at the last pages of the paper only ones having a ridge of type $\langle2,2,n,2,2\rangle$. A routine computation of their growth functions allows to conclude $\frac{\mathrm{d}F_n}{\mathrm{d}t}(1)<0$.
\end{proof}

\begin{proposition}\label{GF_compact}
Let $\mathcal{P}(n)\subset\mathbb H^3$ be a compact Coxeter polyhedron with a ridge of type $\langle2,2,n,2,2\rangle$ for $n$ sufficiently large. Then the function $F_n(t)$ related to the group $G(\mathcal{P}(n))$ is given by
\begin{equation*}
F_n(t) = \frac{(t-1)P(t)}{(t^n-1)Q_{\infty}(t)},
\end{equation*}
where $Q_{\infty}(t)$ is the denominator polynomial associated with the deformed polyhedron $\mathcal{P}_{\infty}$
with a unique  four-valent ideal vertex from Proposition~\ref{GF_ideal}, and $P(t)$ is a product of cyclotomic polynomials and a Salem polynomial. In addition, $P(1) = 0$.
\end{proposition}
\begin{proof}
Denote by $Fin_n:=\{f_\omega(t)\,|\,\omega \in \Omega_*(\mathcal{P}(n))\,\mbox{ such that }\,G(\omega)\, \mbox{is finite}\}$,
and by $Fin_{\infty}:= \{f_\omega(t)\,|\,\omega \in \Omega_*(\mathcal{P}_{\infty})\mbox{ such that }\,G(\omega)\,\mbox{is finite}\}$ where $*\in\{0,1,2\}$. Let $F_n(t) = \frac{P(t)}{Q(t)}$ be given in its virgin form, that means $Q(t)$ is the least common multiple of all polynomials in $Fin_n$. For the corresponding function $F_{\infty}(t)$,  Theorem~\ref{Steinberg_formula} implies that $Q_{\infty}(t)$ is the least common multiple of all polynomials in $Fin_{\infty}$. 

Denote by $e$ the edge of $\mathcal{P}(n)$ undergoing contraction, and let $v_1$, $v_2$ be its vertices. Then the growth function of $Stab(e) \cong D_{n}$ is $f_e(t) = [2][n]$, and the growth function of $Stab(v_i) \cong \Delta_{2,2,n}$ is $f_{v_i}(t) = [2]^2[n]$, $i=1, 2$. The sets $Fin_n$ and $Fin_{\infty}$ differ only by the elements $f_e(t)$, $f_{v_1}(t)$, $f_{v_2}(t)$. 
Furthermore, both sets contain the polynomial $[2]^2$, since the polyhedra $\mathcal{P}(n)$ and 
$\mathcal{P}_{\infty}$ have pairs of edges with right angles along them and stabilizer $D_2$. The comparison of the least common multiples for polynomials in $Fin_n$ and in $Fin_{\infty}$ shows that $Q(t) = Q_{\infty}(t) \cdot [n]$, as claimed.

The assertion $P(1)=0$ follows from the fact that $F_n(1) = 0$ while $\lim_{t\rightarrow 1}\frac{t^n-1}{t-1} = n$. Finally, the polynomial $P(t)$ is a product of cyclotomic polynomials and a Salem polynomial by Theorem~\ref{GF_representation}.
\end{proof}

\begin{theorem}\label{alg_integers}
Let $\mathcal{P}(n)\subset \mathbb H^3$ be a compact Coxeter polyhedron with a ridge $e$ of type $\langle2,2,n,2,2\rangle$ for sufficiently large $n$. Denote by $\mathcal{P}_{\infty}$ the polyhedron arising by contraction of the ridge $e$. Let $\tau_n$ and $\tau_\infty$ be the growth rates of $G(\mathcal{P}(n))$ and $G(\mathcal{P}_\infty)$, respectively. Then $\tau_n < \tau_\infty$ for all $n$, and $\tau_n \rightarrow \tau_\infty$ as $n \rightarrow \infty$. Furthermore,  $\tau_{\infty}$ is a Pisot number.
\end{theorem}
\begin{proof}
The first assertion follows easily from Proposition~\ref{GE_ideal}. We prove that  $\tau_{\infty}$ is a Pisot number by using some number-theoretical properties of growth rates. Consider the growth functions $f_n(t)$ and $f_{\infty}(t)$ of $G(\mathcal{P}(n))$ and $G(\mathcal{P}_{\infty})$, respectively,
together with associated functions $F_n(t) = \frac{1}{f_n(t^{-1})}$ and $F_{\infty}(t)=\frac{1}{f_{\infty}(t^{-1})}$. Then the growth rates $\tau_n$ and $ \tau_\infty$ are the least positive zeros in the interval $(1, +\infty)$ of the functions $F_n(t)$ and $F_{\infty}(t)$.
 
By using Propositions~\ref{GE_ideal}, \ref{GF_ideal} and \ref{GF_compact} in order to represent
the numerator and denominator polynomials of $F_n(t)$ and $F_{\infty}(t)$, one easily derives the equation
\begin{equation}\label{eqpoly1}
\frac{(t-1)P(t)}{(t^n-1)Q_{\infty}(t)} - \frac{t(t-1)P_{\infty}(t)}{Q_{\infty}(t)} = \frac{1}{t^n-1} \left(\frac{t-1}{t+1}\right)^2~.
\end{equation}
For the polynomial $P(t)$, 
we prove that
\begin{equation}\label{eqpoly2}
P(t) = t^{n+1} P_{\infty}(t) - \widetilde{P}_{\infty}(t)
\end{equation}
is a solution to (\ref{eqpoly1}), where $\widetilde{P}_{\infty}(t)$ denotes the reciprocal polynomial of $P_{\infty}(t)$, that is, $\widetilde{P}_{\infty}(t) = t^{\mathrm{deg}\,P_{\infty}(t)} P_{\infty}(t^{-1})$. Since $Q_{\infty}(t)$ is a product of cyclotomic polynomials $\Phi_k(t)$ with $k\geq 2$, one has $Q_{\infty}(t) = \widetilde{Q}_{\infty}(t)=t^{\mathrm{deg}\,Q_{\infty}(t)} Q_{\infty}(t^{-1})$.

\smallskip
Now, replace $P(t)$ in (\ref{eqpoly1}) by its expression from (\ref{eqpoly2})
and simplify each term. This yields
\begin{equation*}
\frac{t(t-1)P_{\infty}(t)}{Q_{\infty}(t)} - \frac{(t-1)\widetilde{P}_{\infty}(t)}{Q_{\infty}(t)} = \left(\frac{t-1}{t+1}\right)^2.
\end{equation*}
By replacing the reciprocal polynomials and by using the fact of Proposition \ref{GF_ideal}, saying
that $\mathrm{deg}\,Q_{\infty}(t) - \mathrm{deg}\,P_{\infty}(t) = 2$, we obtain
\begin{equation}\label{eqpoly3}
\frac{t(t-1)P_{\infty}(t)}{Q_{\infty}(t)} + \frac{t^{-1}(t^{-1}-1)P_{\infty}(t^{-1})}{Q_{\infty}(t^{-1})} = \left(\frac{t-1}{t+1}\right)^2.
\end{equation}
The identity for $F_{\infty}(t)$ as described by Proposition~\ref{GF_ideal} transforms the equation (\ref{eqpoly3}) into $F_{\infty}(t) + F_{\infty}(t^{-1}) = \left(\frac{t-1}{t+1}\right)^2$. Then Proposition~\ref{GE_ideal} provides the equivalent identity $F_n(t)+F_n(t^{-1}) = 0$, which is true by the anti-reciprocity of $F_n(t)$ (see Theorem~\ref{reciprocity}). 

\smallskip
As a consequence,  the relation $P(t) = t^{n+1}P_{\infty}(t) - \widetilde{P}_{\infty}(t)$ holds for $n$ large enough. Since we already know that $P(t)$ is a product of cyclotomic polynomials and a Salem polynomial, Lemma~\ref{pisot} implies that $P_{\infty}(t)$ is a product of cyclotomic polynomials and a Pisot polynomial. Hence, the growth rate $\tau_{\infty}$ is a Pisot number.
\end{proof}

\section{Some final remarks}

\subsection{Deforming L\"{o}bell polyhedra}

The family of L\"{o}bell polyhedra $L(n)$, $n \geq 5$ is described in \cite{Vesnin}. Contracting an edge of $L(5)$, a combinatorial dodecahedron, one obtains the smallest $3$-dimensional right-angled polyhedron with a single ideal four-valent vertex. Contracting all the vertical edges of $L(n)$ as shown in Fig.~\ref{loebell_n} one obtains an ideal right-angled polyhedron $L_{\infty}(n)$, $n\geq 5$. 

\begin{figure}[ht]
\begin{center}
\includegraphics* [totalheight=4cm]{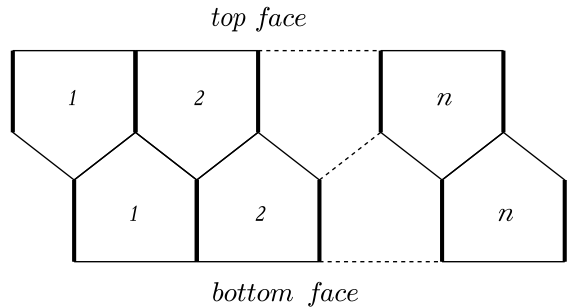}
\end{center}
\caption{L\"{o}bell polyhedron $L(n)$, $n \geq 5$ with one of its perfect matchings marked with thickened edges. Left- and right-hand side edges are identified} \label{loebell_n}
\end{figure}

Note, that contracted edges form a perfect matching of $L(n)$ considered as a three-valent graph. The analogous ideal right-angled polyhedra $L_{\infty}(4)$ and $L_{\infty}(3)$ also exist. Observe that $L_{\infty}(3)$ is a combinatorial octahedron. The growth rate of $L_{\infty}(n)$, $n\geq 3$, belongs to the $(2n)$-th derived set of Salem numbers by Propositions \ref{four_valent_vertex} and \ref{GE_ideal}.

\subsection{Deforming a Lambert cube}

Contracting essential edges of a Lambert cube, one obtains a right-angled polyhedron $\mathcal{R}$. This polyhedron could also be obtained from the Lanner tetrahedron $[3,4,4]$ by means of construction described in \cite{VinbergPotyagailo}. The polyhedron $\mathcal{R}$ is known to have the minimal number of faces among all the right-angled three-dimensional hyperbolic polyhedra of finite volume \cite{ERT}.

\subsection{Finite volume Coxeter polyhedra with an ideal three-valent vertex}

Consider the dodecahedron $\mathcal{D}$ in Fig.~\ref{dodecahedronideal}. It has all but three right dihedral angles. The remaining ones, along the edges incident to a single ideal three-valent vertex, equal $\frac{\pi}{3}$. 

\begin{figure}[ht]
\begin{center}
\includegraphics* [totalheight=4cm]{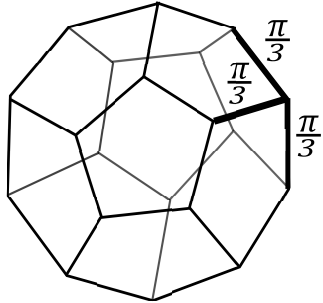}
\end{center}
\caption{The dodecahedron $\mathcal{D}$ with one ideal three-valent vertex} \label{dodecahedronideal}
\end{figure}

The growth function of the corresponding Coxeter group is given by 
\begin{equation*}
f(t) = \frac{(1+t)^3(1+t+t^2)}{9t^4-2t^2-8t+1}=:\frac{Q(t)}{(t-1)\,P(t)}\,,
\end{equation*}
where the polynomial $P(t)$ has integer coefficients. Its reciprocal $\tilde{P}(t)$ is the minimal polynomial of the corresponding growth rate $\tau$. More precisely, $\tilde{P}(t) = 9+9t+7t^2-t^3$ with roots $\tau \approx 8.2269405$ and $\varsigma_{1} = \overline{\varsigma_{2}} \approx -0.6134703+0.8471252 i$. Since $\varsigma_{1}\varsigma_{2} \approx 1.0939668 > 1$, the growth rate $\tau$ of the group $G(\mathcal{D})$ is neither a Salem number, nor a Pisot number.

\bigskip
{\it
Alexander Kolpakov

Department of Mathematics

University of Fribourg

chemin du Mus\'ee 23

CH-1700 Fribourg, Switzerland

aleksandr.kolpakov(at)unifr.ch}

\newpage

\begin{table}[ht]
\begin{center}
\begin{tabular}{|c|c|c|}
\hline & &\\  
Type of $H_1$&  Type of $H_2$&  $\frac{1}{f_1(t)} - \frac{1}{f_2(t)} = $\\[2 ex]
\hline & &\\ 
$\langle2, 2, n, 2, 2\rangle$,&  $\langle2, 2, n+1, 2, 2\rangle$,&  $\frac{t^n(1-t)^3}{(1-t^n)(1-t^{n+1})(1+t)^2}$,\\[2 ex] 
$n\geq 2$& $n\geq 2$& $n\geq 2$\\
\hline & &\\
$\langle2, 2, 2, 2, 3\rangle$&  $\langle2, 2, 3, 2, 3\rangle$&  $\frac{t^2(1-t)}{(1+t)^3(1+t^2)}$\\[2 ex]
\hline & &\\
$\langle2, 2, 3, 2, 3\rangle$&  $\langle2, 2, 4, 2, 3\rangle$&  $\frac{t^3(1-t)}{(1+t)^3(1-t+t^2)(1+t+t^2)}$\\[2 ex]
\hline & &\\
$\langle2, 2, 4, 2, 3\rangle$&  $\langle2, 2, 5, 2, 3\rangle$&  $\frac{t^4(1-t)(1-t+t^2)(1+t+t^2)}{(1+t)^3(1+t^2)(1-t+t^2-t^3+t^4)(1+t+t^2+t^3+t^4)}$\\[2 ex] 
\hline & &\\ 
$\langle2, 2, 2, 2, 4\rangle$&  $\langle2, 2, 3, 2, 4\rangle$&  $\frac{t^2(1-t)(1+t^2)}{(1+t)^3(1-t+t^2)(1+t+t^2)}$\\[2 ex]
\hline & &\\ 
$\langle2, 2, 2, 2, 5\rangle$&  $\langle2, 2, 3, 2, 5\rangle$&  $\frac{t^2(1-t)(1+2 t^2+t^3+2 t^4+t^5+2 t^6+t^7+2 t^8+t^{10})}{(1+t)^3 (1+2 t^2+3 t^4+3 t^6+3 t^8+2 t^{10}+t^{12})}$\\[2 ex]
\hline & &\\ 
$\langle2, 3, 2, 2, 3\rangle$&  $\langle2, 3, 3, 2, 3\rangle$&  $\frac{t^2(1-t)}{(1+t)(1+t^2)(1+t+t^2)}$\\[2 ex]
\hline & &\\ 
$\langle2, 3, 3, 2, 3\rangle$&  $\langle2, 3, 4, 2, 3\rangle$&  $\frac{t^3(1-t)}{(1+t)^3(1+t^2)(1-t+t^2)}$\\[2 ex]
\hline & &\\ 
$\langle2, 3, 4, 2, 3\rangle$&  $\langle2, 3, 5, 2, 3\rangle$&  $\frac{t^4(1-t)}{(1+t)^3(1+t^2)(1-t+t^2-t^3+t^4)}$\\[2 ex] 
\hline & &\\ 
$\langle2, 3, 2, 2, 4\rangle$&  $\langle2, 3, 3, 2, 4\rangle$&  $\frac{t^2(1-t)(1+t+t^2+t^3+t^4)}{(1+t)^3(1+t^2)(1-t+t^2)(1+t+t^2)}$\\[2 ex]
\hline 
\end{tabular} 
\end{center}
\caption{Table for Proposition \ref{GE_growth}} \label{tabular_for_GE_growth1}
\end{table}

\begin{table}[ht]
\begin{center}
\begin{tabular}{|c|c|c|}
\hline & &\\  
Type of $H_1$&  Type of $H_2$&  $\frac{1}{f_1(t)} - \frac{1}{f_2(t)} = $\\[2 ex]
\hline & &\\ 
$\langle2, 3, 2, 2, 5\rangle$&  $\langle2, 3, 3, 2, 5\rangle$&  $\frac{t^2(1-t)(1+t^2)(1+t^3+t^6)}{(1+t)^3(1+3 t^2+5 t^4+6 t^6+6 t^8+5 t^{10}+3 t^{12}+t^{14})}$\\[2 ex]
\hline & &\\ 
$\langle2, 4, 2, 2, 4\rangle$&  $\langle2, 4, 3, 2, 4\rangle$&  $\frac{t^2(1-t)}{(1+t)^3(1-t+t^2)}$\\[2 ex]
\hline & &\\ 
$\langle2, 4, 2, 2, 5\rangle$&  $\langle2, 4, 3, 2, 5\rangle$&  $\frac{t^2(1-t)(1+t^2)(1+t^3+t^6)}{(1+t)^3 (1-t+t^2) (1-t+t^2-t^3+t^4) (1+t+t^2+t^3+t^4)}$\\[2 ex]
\hline & &\\ 
$\langle2, 5, 2, 2, 5\rangle$&  $\langle2, 5, 3, 2, 5\rangle$&  $\frac{(1-t)t^2(1+t^2+2 t^3-t^4+2 t^5+t^6+t^8)}{(1+t)^3(1-t+t^2)(1-t+t^2-t^3+t^4)(1+t+t^2+t^3+t^4)}$\\[2 ex]
\hline 
\end{tabular} 
\end{center}
\caption{Table for Proposition \ref{GE_growth} (continuation)} \label{tabular_for_GE_growth2}
\end{table}

\begin{table}[ht]
\begin{center}
\begin{tabular}{|c|c|c|c|c|}
\hline \multirow{2}{*}{Vertex group $Stab(v)$}& \multicolumn{3}{|c|}{Its Coxeter exponents}&  \multirow{2}{*}{Quantity $\frac{\mathrm{d}g_v}{\mathrm{d}t}(1)$}\\
\cline{2-4} & $m_1$&  $m_2$&  $m_3$&  \\ 
\hline  $\Delta_{2,2,n}$, $n\geq 2$& 1&  1&  $n-1$&  $-\frac{1}{8}\left(1-\frac{1}{n}\right)$\\
\hline  $\Delta_{2,3,3}$&  1&  2&  3&  $-\frac{1}{8}$\\ 
\hline  $\Delta_{2,3,4}$&  1&  3&  5&  $-\frac{5}{32}$\\ 
\hline  $\Delta_{2,3,5}$&  1&  5&  9&  $-\frac{3}{16}$\\ 
\hline 
\end{tabular} 
\end{center}
\caption{Table for Proposition \ref{GF_ideal}. The column on the right follows from Theorem~\ref{GF_representation}, formula (\ref{eqParry2})} \label{tabular_for_GF_ideal}
\end{table}

\begin{figure}[ht]
\begin{center}
\includegraphics* [totalheight=7cm]{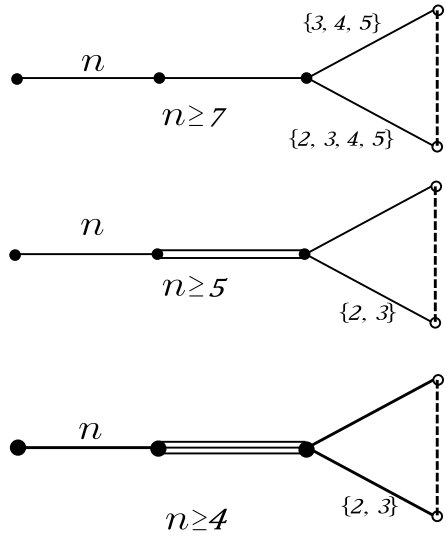}
\end{center}
\caption{Prisms that admit contraction of an edge: the first picture} \label{prismcontraction1}
\end{figure}

\begin{figure}[ht]
\begin{center}
\includegraphics* [totalheight=3cm]{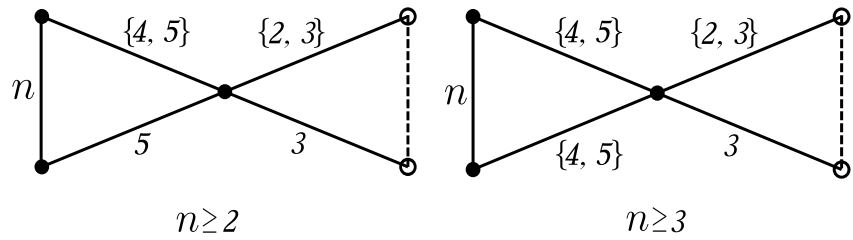}
\end{center}
\caption{Prisms that admit contraction of an edge: the second picture} \label{prismcontraction2}
\end{figure}

\begin{figure}[ht]
\begin{center}
\includegraphics* [totalheight=6cm]{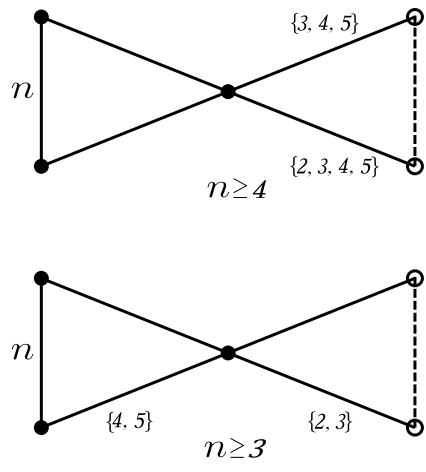}
\end{center}
\caption{Prisms that admit contraction of an edge: the third picture} \label{prismcontraction3}
\end{figure}

\end{document}